\documentclass[a4paper,11pt]{article}
\usepackage[top=2.5cm,bottom=2.5cm,left=2.2cm,right=2.2cm]{geometry}
\usepackage[T1]{fontenc}
\usepackage{lmodern}
\usepackage{amsmath,amsfonts,amssymb,amsthm,mathtools}
\usepackage{indentfirst}
\usepackage{enumerate}
\usepackage{microtype,needspace}
\usepackage{cite}
\usepackage{hyperref}
\usepackage{bookmark}
\hypersetup{
  colorlinks=true,linkcolor=blue,citecolor=blue,urlcolor=blue,
  pdftitle={Positive discrepancy of graphs far from Turan graphs},
  pdfauthor={},
  pdfsubject={Positive discrepancy and second eigenvalues of graphs},
  pdfkeywords={positive discrepancy, Turan graph, second eigenvalue, graph stability}
}
\makeatletter
\renewcommand{\@seccntformat}[1]{%
  {\normalsize\csname the#1\endcsname.}\hspace{.5em}}
\makeatother

\numberwithin{equation}{section}
\newtheorem{thm}{Theorem}[section]
\newtheorem{lemma}[thm]{Lemma}
\newtheorem{prop}[thm]{Proposition}
\newtheorem{cor}[thm]{Corollary}

\newtheorem{conj}{Conjecture}
\theoremstyle{definition}

\theoremstyle{remark}

\DeclareMathOperator{\tr}{tr}
\DeclareMathOperator{\rank}{rk}
\DeclareMathOperator{\spanof}{span}
\DeclareMathOperator{\ed}{ed}
\DeclareMathOperator{\disc}{disc}
\DeclareMathOperator{\diag}{diag}
\newcommand{\eps}{\varepsilon}
\newcommand{\R}{\mathbb R}
\newcommand{\normF}[1]{\lVert #1\rVert_F}
\newcommand{\normop}[1]{\lVert #1\rVert_{\mathrm{op}}}
\newcommand{\norm}[1]{\lVert #1\rVert_2}
\newcommand{\ip}[2]{\langle #1,#2\rangle}
\allowdisplaybreaks[1]
\begin{document}
\title{\bfseries Positive discrepancy of graphs far from Tur\'an graphs}
\author{Leyou Xu, 
Bo Zhou\\
\footnotesize South China Normal University, Guangzhou 510631, China}
\date{}
\maketitle
\vspace{-2.5em}

\begin{abstract}
We prove that, for every $\eps>0$, any $n$-vertex graph that needs at
least $\eps n^2$ edge changes to become a Tur\'an graph
has positive discrepancy at least $c_\eps n^{5/4}$. 
Consequently, every such regular graph has second eigenvalue at least $c'_\eps n^{1/4}$. 
These results prove two conjectures of R\"aty, Sudakov and Tomon. 

\smallskip
\noindent\textbf{Keywords:} positive discrepancy, Tur\'an graph, second eigenvalue

\end{abstract}

\section{Introduction}

All graphs in this paper are finite, simple and undirected. Let $G$ be a graph with vertex set 
$V(G)$ and edge set $E(G)$. Let 
$e(G)=|E(G)|$. For $U\subseteq V(G)$,  $G[U]$ denotes the subgraph of $G$ induced by
$U$.
For a graph $G$ with $n$ vertices and  edge density  $p=p(G)=\frac{e(G)}{\binom n2}$,  the {\it positive discrepancy} is defined as 
\[
 \disc^{+}(G)=\max_{U\subseteq V(G)}\disc_G(U), 
\]
where 
\[
\disc_G(U)=e(G[U])-p\binom{|U|}{2}.
\]

For graphs $G,H$ on $n$ vertices, let $\ed(G,H)$ be the minimum number of edge changes (additions and deletions) needed to turn $G$ into a graph that is isomorphic to $H$.  We say that $G$ is \emph{$\eps$-far from} $H$ if $\ed(G,H)\ge\eps n^2$. Otherwise, $G$ is \emph{$\eps$-close to} $H$.
The \emph{Tur\'an\ graph} $T_r(n)$, where $1\le r\le n$, is the complete
$r$-partite graph whose $r$ nonempty parts have sizes differing by at most
one. In particular, $T_1(n)$ is the empty graph on $n$ vertices and $T_n(n)$ is the 
complete graph on $n$ vertices. We say that $G$ is $\eps$-far from every Tur\'an\ graph if $G$ is $\eps$-far from $T_r(n)$ for every integer $r=1,\dots,n$.

\begin{thm}\cite[Theorem 1.2]{RST}
There exists $c > 0$, and for every $\epsilon \in (0,1/2)$ there exists $c_1 > 0$ such that for every sufficiently large $n$, and any graph  $G$ on $n$ vertices of average degree $d$,
\[
\disc^{+}(G) \ge
\begin{cases}
c\, d^{1/2} n & \text{if } d \le n^{2/3}, \\[4pt]
c\, \dfrac{n^{2}}{d} & \text{if } n^{2/3} \le d \le n^{4/5}, \\[8pt]
c_1\, \dfrac{d^{1/4} n}{\log n} & \text{if } n^{4/5} \le d \le \left(\dfrac{1}{2}-\epsilon\right) n.
\end{cases}
\]
\end{thm}

The only known examples of $n$-vertex graphs with positive discrepancy $O(n)$ are the Tur\'an graphs $T_r(n)$, and graphs that are $O(\frac{1}{n})$-close to a Tur\'an graph. This suggests that the positive discrepancy of $G$ might be substantially larger if $G$ is far from being a Tur\'an graph. de Caen~\cite{Ca} constructed a $d$-regular graph on $n$ vertices with $d=\Theta(n)$,, which is $\Omega(1)$-far from every Tur\'an graph and has positive discrepancy $O(n^{5/4})$. Based on these observations, R\"aty, Sudakov and Tomon \cite{RST} proposed the following conjecture.

\begin{conj}\cite[Conjecture~8.1]{RST} \label{c1}
Let $\eps>0>0$. There exists $c>0$ such that if $G$ is an $n$-vertex graph that is $\eps>0$-far from every Tur\'an graph, then
\[
\disc^{+}(G)\ge cn^{5/4}.
\]
\end{conj}

For a graph $G$, we write $A(G)$ for its adjacency matrix and order its real eigenvalues as $\lambda_1(G)\ge\cdots\ge\lambda_n(G)$, counting with multiplicity. We call $\lambda_2(G)$ the second eigenvalue of $G$.

For a $d$-regular graph $G$ on $n$ vertices, it is known that 
\[
\disc^{+}(G) \le \frac{\lambda_2(G)}{2}n + d.
\]
R\"aty, Sudakov and Tomon \cite{RST} proposed another conjecture.

\begin{conj}\cite[Conjecture~8.2]{RST} \label{c2}
Let $\eps>0>0$. There exists $c>0$ such that if $G$ is a regular $n$-vertex graph that is $\eps>0$-far from every Tur\'an graph, then its second largest eigenvalue is at least $cn^{1/4}$.
\end{conj}

Zhang \cite{Zhang}  confirmd  Conjecture \ref{c2} up to an $n^{-\eps}$ factor.
Jin, Milojevi\'c, Tomon and Zhang~\cite[Theorems~1.4 and~1.6]{JMTZ} proved that, for every fixed $0<\gamma<\frac{1}{4}$ and every fixed $\delta>0$, a sufficiently large graph with least eigenvalue at least $-n^\gamma$ can be made a disjoint union of complete graphs by at most $\delta n^2$ edge changes, and a regular graph with $\lambda_2\le n^\gamma$ isclose to a Tur\'an graph. 
These statements with fixed $\gamma<\frac{1}{4}$ do not provide a positive coefficient of $n^{1/4}$ when the number of allowed edge additions and deletions is fixed.

In this paper we resolve both conjectures in the affirmative. 
The results of this paper are the following Theorem \ref{thm:discrepancy} and Corollary \ref{cor:eigenvalue}, confirming Conjectures \ref{c1} and \ref{c2},  respectively.

\begin{thm}\label{thm:discrepancy}
For every $\eps>0$, there is a constant $c>0$ such that for  any $n$-vertex graph $G$ that is $\eps$-far from every Tur\'an graph,  
 $\disc^{+}(G)\ge c n^{5/4}$.
\end{thm}

\begin{cor}\label{cor:eigenvalue}
For every $\eps>0$, there is a constant $c'>0$ such that for any regular $n$-vertex graph $G$ that is $\eps$-far from every Tur\'an graph, $\lambda_2(G)\ge c' n^{1/4}$.
\end{cor}

Neither result requires connectivity, and Theorem~\ref{thm:discrepancy} does not require regularity. The regular graphs constructed by de Caen, as discussed in~\cite[Section~8]{RST}, show that the exponents $\frac{5}{4}$ and $\frac{1}{4}$ cannot be improved.

Our spectral argument builds on the compression method of~\cite{Zhang,JMTZ}. In particular, we supplement the quadratic recursion in~\cite[Lemma~6.3]{JMTZ} with a cubic recursion and an explicit control of matrix errors. Following~\cite{GW,BHT}, we use random hyperplane rounding to obtain a matrix inequality from positive discrepancy. A key point is that we retain a squared Frobenius norm error term in this inequality: rather than discarding the error, we carry it through the subsequent spectral argument.

The paper is organized as follows.
In Section~\ref{sec:prelim}, we give  the notation and the standard facts that will  be used in the proofs.
In Section~\ref{sec3}, we give  the structural theorem behind the proof, along with a related bound and a balanced clique partition result. 
In Section~\ref{sec:main} we prove the main results.

\section{Preliminaries}\label{sec:prelim}

Let $G$ be a graph with vertex set $V(G)$ and edge set $E(G)$.
For $v\in V(G)$, we denote by $N_G(v)$ the neighborhood of $v$ in $G$, and by $d_G(v)$ the degree of $v$ in $G$. Let $N_G[v]:=\{v\}\cup N_G(v)$ be the closed neighborhood of $v$.
Let $d(G)$ denote the average degree of $G$, i.e., $d(G)=\frac{2e(G)}{n}$. 
For $S\subseteq V(G)$, $G[S]$ denotes the subgraph induced by $S$. We write $G-S=G[V(G)\setminus S]$ and $e_G(S)=e(G[S])$. 
For disjoint vertex sets $S,T$, let $e_G(S,T)$ denote the number of edges with one
endpoint in each set. When $S,T\ne \emptyset$, the \emph{edge density} of $S$ and $T$ is defined by
\[
 d_G(S,T)=\frac{e_G(S,T)}{|S||T|}.
\]
We omit the subscripts whenever the graph is clear from context.
Let $\overline G$ denote the complement of $G$.

A \emph{union of cliques} is a vertex-disjoint union of complete graphs, with single-vertex components allowed. It is \emph{balanced} if its component orders differ by at most one. The complement of a balanced union of cliques is a Tur\'an\ graph. 

For an \emph{induced $P_3$}, the \emph{middle vertex} is its unique vertex adjacent to the other two, and the \emph{endpoints} are the other two vertices. 
For a graph $G$, let \[
p_3(G)=|\{U\subseteq V(G):G[U]\cong P_3,|U|=3\}|.
\]
A simple fact, which we shall use in this paper, is that a graph is a union of cliques if and only if it has no induced $P_3$, see~\cite{AS}.

The symmetric difference of two sets $S$ and $T$ is $S\mathbin\triangle T$. 
Let $G$ and $H$ be two graphs on vertex set $V$. 
Exactly $|E(G)\mathbin\triangle E(H)|$ changes turn $G$ into $H$. Thus $\ed(G,H)$ is the minimum, over all bijections from $V(G)$ to $V(H)$, of the number of edges on which $G$ and the image of $H$ under that bijection differ. Taking complements of both graphs preserves every symmetric difference, so $\ed(\overline{G},\overline{H})=\ed(G,H)$. 
A sequence of editing procedures costs at most the sum of their individual costs, even if a pair is changed more than once. Every edge addition or deletion changes the degrees
of exactly two vertices by one. In particular, 
\begin{equation}\label{eq:degree-edit}
 \sum_{v\in V}|d_G(v)-d_H(v)|
 \le2|E(G)\mathbin\triangle E(H)|.
\end{equation}

We use the Euclidean inner product $\ip{\mathbf{x}}{\mathbf{y}}=\mathbf{x}^\top \mathbf{y}$ and norm $\norm{\mathbf{x}}$ for vectors.
Let $\mathbf{e}_i$ be the $i$th coordinate vector, $\mathbf 1$ the all ones vector, and $\mathbf{0}$ the zero vector. Let $I$ denote the identity matrix and $J$ the all-one matrix.
For $U\subseteq[n]=\{1,\dots,n\}$, its characteristic vector $\mathbf 1_U$ has coordinate
$1$ on $U$ and $0$ elsewhere. The notation
$\mathbf 1^\perp=\{\mathbf{x}:\ip{\mathbf{x}}{\mathbf 1}=0\}$ denotes the orthogonal complement of $\mathbf 1$. 
For a symmetric matrix $M$, $M[U]$ is its principal submatrix on rows and columns in $U$, and $\diag(\mathbf{x})$ is the diagonal matrix with diagonal $\mathbf{x}=(x_1,\ldots,x_n)$.

Let $M$ be a symmetric matrix. The rank of $M$ is denoted by $\rank M$. 
We write $M\succeq0$ if $M$ is \emph{positive semidefinite}. The \emph{Frobenius norm} and the \emph{operator norm} of $M$ are denoted by $\normF{M}$  and $\normop{M}$, respectively. 
For two symmetric matrices $M$ and $M'$, we write $M\succeq M'$ if $M-M'\succeq0$, and define their inner product by $\ip{M}{M'}=\sum_{i,j}M_{ij}M'_{ij}=\tr(MM')$.
Moreover, if $M,M'\succeq0$, then $\ip{M}{M'}\ge 0$.

Let $\lambda_1(M)\ge \dots \ge\lambda_n(M)$ denote the eigenvalues of an $n\times n$ symmetric matrix $M$, and let  $\mathbf{z}_1,\ldots,\mathbf{z}_n$ be the orthonormal eigenvectors corresponding to $\lambda_1(M), \dots,\lambda_n(M)$, respectively. Then
$M=\sum_{i=1}^n\lambda_i(M)\mathbf{z}_i \mathbf{z}_i^\top$. 
For a scalar $t$, let $t_+=\max\{t,0\}$. The \emph{positive and negative
parts} of a symmetric matrix are $M_+=\sum_{i=1}^n\lambda_i(M)_+\mathbf{z}_i\mathbf{z}_i^\top$ and $M_-=\sum_{i=1}^n(-\lambda_i(M))_+\mathbf{z}_i\mathbf{z}_i^\top$.
It is clear that $M=M_+-M_-$, $M_+,M_-\succeq0$, and $\ip{M_+}{M_-}=0$.
A matrix $M\succeq0$ of rank at most $r$ is the Gram matrix of vectors $\mathbf{z}_1,\ldots,\mathbf{z}_n\in\R^r$, so $M_{ij}=\ip{\mathbf{z}_i}{\mathbf{z}_j}$.
Consequently, $|M_{ij}|\le\sqrt{M_{ii}M_{jj}}$ and $\sum_i\norm{\mathbf{z}_i}^2=\tr M\le r\normop{M}$.

For a subspace $W\subseteq\R^n$, let $P_W$ be its orthogonal projection and denote the $W$-trace of $M$ by $\tr_W M=\tr(P_WM)$. We shall use
$\tr P_W=\dim W$, $\normF{P_W}=\sqrt{\dim W}$ and $\tr_W(\mathbf{z}\mathbf{z}^\top)=\norm{P_W\mathbf{z}}^2$.
The \emph{Hadamard} product of two matrices $M$ and $M'$ is given by $(M\circ M')_{ij}=M_{ij}M'_{ij}$. For vectors, the Hadamard product is defined analogously by $(\mathbf{x}\circ \mathbf{y})_i=x_iy_i$.
The power $M^{\circ q}$ is entrywise, not the ordinary matrix power.

The following result is a consequence of Weyl's inequality~\cite[Theorem 4.3.1]{HJ}.
\begin{lemma}\label{weyl}
For symmetric matrices $M$ and $M'$, $\lambda_n(M-M') \ge \lambda_n(M) - \|M'\|_{\mathrm{op}}$.
\end{lemma}

The following is a direct consequence of~\cite[Lemma~3.2]{GW}. 
A quadratic error term of the same type appears in the argument in \cite[Lemma~9.4]{BHT}, although that lemma is concerned with graph energy rather than the discrepancy certificate considered here. 
The result we stated follows by completing the diagonal of $X$ to all ones, applying random-hyperplane rounding, and estimating only the off-diagonal error terms. We include the short argument below.

\begin{lemma}\label{lem:rounding}
Let $M$ be a symmetric $n\times n$ matrix with zero diagonal and
$|M_{ij}|\le1$, and let $\beta=\max_{\sigma\in\{-1,1\}^n}\sigma^\top M\sigma$.
For every $X\succeq0$ with $X_{ii}\le1$ for $i=1,\dots,n$,
$\ip{M}{X}-\normF{X}^2\le\frac\pi2\beta$.
\end{lemma}
\begin{proof}
Let $\widehat{X}=X+\operatorname{diag}(1-X_{11},\dots,1-X_{nn})$. Then $\widehat{X}\succeq 0$ and $\widehat{X}_{ii}=1$ for every $i$. Hence there exists a random sign vector $\sigma\in\{-1,1\}^n$ such that
$\mathbb{E}(\sigma_i\sigma_j)=\frac{2}{\pi}\arcsin X_{ij}$ for $i\ne j$.
Since $M$ has zero diagonal, $|M_{ij}|\le 1$ and $|\arcsin u-u|\le u^2$ for $|u|\le 1$, we have
\begin{align*}
\beta&\ge\mathbb{E}(\sigma^\top M\sigma)=\frac{2}{\pi}\sum_{i\ne j}M_{ij}\arcsin X_{ij}\\
&\ge\frac{2}{\pi}\left(\langle M,X\rangle-\sum_{i\ne j}X_{ij}^2\right)\ge\frac{2}{\pi}\left(\langle M,X\rangle-\|X\|_F^2\right).
\end{align*}
This proves the claimed inequality.
\end{proof}

For the following lemma, the first part can be found in~\cite{Higham}, and the second part follows from the projection and~\cite[Sections~2.3, 3.1 and~6.3.3]{PB}.

\begin{lemma}\label{lem:positivepart}
For every symmetric matrix $M$, let $g(M)=\frac14\normF{M_+}^2$. Then 
$g(M)=\max_{X\succeq0}\bigl(\ip{M}{X}-\normF{X}^2\bigr)$,
and the maximizer is uniquely $X=\frac{1}{2}M_+$.
The function $g$ is continuously differentiable, with
$\nabla g(M)=\frac{1}{2}M_+$. In particular, for every symmetric matrix $E$, the directional derivative of $g$ at $M$ in direction $E$ is $\frac{1}{2}\langle M_+,E\rangle$, even when $M$ has zero eigenvalues.
\end{lemma}

The following standard spectral results will be used in the regular case. 
The first bound for the second eigenvalue follows from \cite[Theorem 3.5.1]{BH}, and the identity for the complement is from \cite[Section~1.3.2]{BH}. The second bound can be found in~\cite{Pet}.
The discrepancy estimate follows from~\cite[Exercise~3.3(i)]{BH}.

\begin{lemma}\label{lem:regular}
Let $G$ be a noncomplete $d$-regular graph of order $n\ge2$.
Then $\lambda_2(G)\ge0$,
$\lambda_n(\overline G)=-1-\lambda_2(G)$, and
$\disc^+(G)\le\frac n8\left(\lambda_2(G)+\frac{d}{n-1}\right)$.
If $\overline G$ is not a union of cliques, then
$\lambda_2(G)\ge\sqrt2-1$ and $\disc^+(G)\le\frac{2+\sqrt2}{8}\,n\lambda_2(G)$.
\end{lemma}

%

\section{Turning graphs into unions of cliques}\label{sec3}

In this section, we establish a structural theorem showing that a graph can be made a union of cliques by few edge additions and deletions, assuming a suitable spectral condition. 
We also obtain a related upper bound on the minimum number of edge additions and deletions, and a result on balanced clique partitions.

\subsection{A structural theorem via spectral conditions}

Let $\delta\in \R$ with $0<\delta<1/2$. 
A pair of disjoint nonempty vertex sets is \emph{$\delta$-empty} if its edge density is at most $\delta$, and \emph{$\delta$-full} if its density is at least $1-\delta$.
Every other pair is called \emph{$\delta$-exceptional}.

\begin{lemma}\label{lem:obstruction}
Let $0<\delta<1/16$ and $s\ge0$. Suppose a graph $G$ satisfies $\lambda_n(G)\ge-s$. There cannot be three disjoint vertex sets $X,Y,Z$ with $|X|=|Y|=|Z|>6s$, such that $(X,Y)$ and $(Y,Z)$ are $\delta$-full and $(X,Z)$ is $\delta$-empty.
\end{lemma}

\begin{proof}
Let $t=|X|=|Y|=|Z|$.
Suppose these sets exist. 
We have $e(X,Y)\ge (1-\delta)t^2$, $e(Y,Z)\ge (1-\delta)t^2$ and $e(X,Z)\le \delta t^2$. 
Let $\mathbf{x}$ be a vector with \[
x_u=\begin{cases}
1&\mbox{ if } u\in X,\\
-2&\mbox{ if }u\in Y,\\
1&\mbox{ if }u\in Z,\\
0&\mbox{ otherwise.}
\end{cases}
\] 
Then, as for $W\in \{X,Y,Z\}$, $e(W)\le {t\choose 2}\le \frac{1}{2} t^2$, we have
\begin{align*}
\mathbf{x}^\top A(G)\mathbf{x}&=2e(X)+8e(Y)+2e(Z)-4e(X,Y)-4e(Y,Z)+2e(X,Z)\\
&\le t^2+4t^2+t^2-4(1-\delta)t^2-4(1-\delta)t^2+2\delta t^2 \\
&=(-2+10\delta)t^2\\
&<-t^2.
\end{align*}
Note that $\mathbf{x}^\top \mathbf{x}=6t$.
So by Rayleigh's principle, \[
\lambda_n\le \frac{\mathbf{x}^\top A(G)\mathbf{x}}{\mathbf{x}^\top \mathbf{x}}<\frac{-t^2}{6t}=-\frac{t}{6}<-s,
\]
a contradiction. 
\end{proof}

For a graph $G$, let $\ed_c(G)$ be the minimum number of edge additions and deletions needed to turn $G$ into a union of cliques. 
Alon and Shapira~\cite{AS} proved a polynomial bound on $\operatorname{ed}_c(G)$ in terms of $p_3(G)$, see also~\cite[Lemma~8.4]{JMTZ}.
We give a direct proof that yields a bound in terms of $\sqrt{p_3(G)}$.

\begin{lemma}\label{lem:editing}
For a graph $G$ of order $n$, $\ed_c(G)\le 3\sqrt{np_3(G)}$.
\end{lemma}
\begin{proof}
If $p_3(G)=0$, then $G$ is already a union of cliques, so $\ed_c(G)=0\le 3\sqrt{np_3(G)}$. 
Suppose that $p_3(G)\ge 1$ and let $b=\sqrt{\frac{p_3(G)}{n^3}}$. 
Let $G_1$ be the subgraph of $G$ induced by the vertices not belonging to any isolated clique.
Then $p_3(G_1)= p_3(G)$. 
For $v\in V(G_1)$, let $p_{3}^{(1)}(v)$ denote the number of sets inducing a $P_3$ in $G_1$ and containing $v$. 

If $e(G_1) \le bn^2$, then deleting all edges of $G_1$ yields a union
of cliques, and the number of edges deleted is at most $bn^2 < 3\sqrt{np_3(G)}$, as desired.

Otherwise, we proceed iteratively. Suppose at some step we have a graph
$G_i$ with $e(G_i) > bn^2$. For $v\in V(G_i)$, let $p_{3}^{(i)}(v)$ denote the number of sets inducing a $P_3$ in $G_i$ and containing $v$. 
Then $\sum_{v \in V(G_i)} p_3^{(i)}(v) = 3p_3(G_i) \le 3p_3(G)$
and $\sum_{v \in V(G_i)} \bigl(d_{G_i}(v)+1\bigr) = 2e(G_i) + |V(G_i)| > 2bn^2$.
Averaging the quantities $\frac{p_3^{(i)}(v)}{d_{G_i}(v)+1}$ over all vertices $v$, we obtain $\frac{\sum_{v \in V(G_i)} p_3^{(i)}(v)}
     {\sum_{v \in V(G_i)} \bigl(d_{G_i}(v)+1\bigr)}$.
Hence there exists a vertex $u_i \in V(G_i)$ such that
\[
\frac{p_3^{(i)}(u_i)}{d_{G_i}(u_i)+1}
\le \frac{\sum_{v \in V(G_i)} p_3^{(i)}(v)}
        {\sum_{v \in V(G_i)} \bigl(d_{G_i}(v)+1\bigr)}
\le \frac{3p_3(G)}{2bn^2} = \frac{3}{2}bn.
\]
Let $C_i = N_{G_i}[u_i]$. We add every missing edge inside $C_i$ and delete every edge between $C_i$ and $V(G_i)\setminus C_i$. The total number of edge additions and deletions is exactly $p_3^{(i)}(u_i)$. We then let $G_{i+1} = G_i - C_i$. One has $p_3(G_{i+1}) \le p_3(G_i)\le p_3(G)$.

This process terminates at some graph $G_t$ with $e(G_t) \le bn^2$.
Deleting all edges of $G_t$ yields a union of cliques. Thus, the total number of edges deleted or added is at most \[
\sum_{i=1}^{t-1} p_3^{(i)}(u_i) + e(G_t)
\le \frac{3}{2}bn \sum_{i=1}^{t-1} \bigl(d_{G_i}(u_i)+1\bigr) + bn^2
\le \frac{5}{2}bn^2\le 3\sqrt{np_3(G)}.
\]
So $\ed_c(G)\le 3\sqrt{np_3(G)}$.
\end{proof}

We next record an auxiliary proposition, which will be used in the proof of the main result in this subsection. For readability, the proof of the following proposition is deferred to Appendix~\ref{app:spectrum}.  An explicit sufficient choice of $N^*(\tau)$ is also given there.


\begin{prop}\label{prop:lowrank}
Let $\tau\in(0,1)$, $a=\frac{\tau}{1000}$, and $\kappa=(\frac{\tau}{1024})^4$. There exists an integer $N^*(\tau)$ 
such that if $n\ge N^*(\tau)$, $A$ is a symmetric $0$-$1$ matrix with zero
diagonal, $R\succeq0$ satisfies $R_{ii}\le2$ for $i=1,\dots,n$ and $\|R\|_F\le n^{2/3}$, and $C=A+R$ satisfies $\lambda_n(C)\ge-a n^{1/4}$, then there is a matrix $B\succeq0$ with
$\rank B\le4\kappa^{-2}$, $\|B\|_{\mathrm{op}}\le2n$ and $\|C-B\|_F^2\le\tau n^2$.
\end{prop}

We will also use the following form of~\cite[Lemma~8.2]{JMTZ}.
In their proof, the approximation error may be taken as $\frac{\delta^2n^2}{100}$, the exceptional set $V_0$ accounts for the rounding of block sizes, and the constant $K_{r,\delta}$ is chosen larger, if necessary, to bound both the number of parts and the required order threshold.
The required approximation error in Lemma~\ref{lem:partition} depends on $\delta$ but not on $r$, while the dimension affects only the constant $K_{r,\delta}$. This order of dependence allows the parameters below to be chosen successively.

\begin{lemma}\label{lem:partition}
Let $r\ge1$ be an integer and $0<\delta<\frac{1}{16}$. There is an integer
$K_{r,\delta}$ such that if $n\ge K_{r,\delta}$ and $A(G)$ has an approximation
$B\succeq0$ satisfying $\rank B\le r$, $\normop{B}\le2n$ and $\normF{A(G)-B}^2\le\frac{\delta^2}{100}n^2$, 
then there is a partition
$V(G)=V_0\cup V_1\cup\cdots\cup V_K$ such that
$|V_0|\le\delta n$, $|V_1|=\cdots=|V_K|=m$, $\frac1\delta\le K\le K_{r,\delta}$ and $m\ge\frac{n}{K_{r,\delta}}$.
The set $V_0$ may be empty. 
Among the unordered pairs $\{V_i,V_j\}$ with $1\le i<j\le K$, at most $\delta K^2$ are $\delta$-exceptional.
\end{lemma}

We are now ready to show the main result of this subsection. 

\begin{thm}\label{thm:structure}
For every $\eta\in(0,1)$, there is an integer $N_\eta\ge2$, 
such that for every $n$-vertex graph $G$ with $n\ge N_\eta$, if there exists a real
symmetric matrix $R$ such that $R\succeq0$, $R_{ii}\le2$ for $i=1,\dots,n$, $\normF{R}\le n^{2/3}$ and $\lambda_n(A(G)+R)\ge-\frac{\eta^4}{4\cdot10^9}n^{1/4}$, then $\ed_c(G)\le \eta n^2$.
\end{thm}

\begin{proof}
Let $\delta=\frac{\eta^2}{100}$, $\xi=\frac{\delta^2}{100}$ and $\tau=\frac{\xi}{4}$.
Then $0<\delta<\frac{1}{16}$ and $\tau\in(0,1)$. Let $\kappa=(\frac{\tau}{1024})^4$, $r=\lceil4\kappa^{-2}\rceil$ and 
$a=\frac{\tau}{1000}$.
Let $N^*(\tau)$ be an integer from Proposition~\ref{prop:lowrank}, and let $K_{r,\delta}$ be an integer from Lemma~\ref{lem:partition}. Let
\begin{equation}\label{eq:structure-threshold}
 N_\eta=\max\left\{
  N^*(\tau),\  (13K_{r,\delta})^3,\ \left\lceil\frac{8}{\xi^{3/2}}\right\rceil \right\}.
\end{equation}

Let $R$ be a matrix as in the statement of the theorem and $C=A(G)+R$. Then $\lambda_n(C)\ge -an^{1/4}$.
Proposition~\ref{prop:lowrank} gives a matrix $B\succeq0$ of rank at
most $r$, with $\normop{B}\le2n$ and $\normF{C-B}^2\le\tau n^2$. 
Then
\[
 \normF{A(G)-B}^2\le2\normF{C-B}^2+2\normF{R}^2\le2\tau n^2+2n^{4/3}\le\xi n^2=\frac{\delta^2}{100}n^2.
\]
Thus, all hypotheses of Lemma~\ref{lem:partition} hold. Let $V(G)=V_0\cup V_1\cup \dots \cup V_K$ be a partition with $|V_i|=m$ for $i=1,\dots,K$, where $m\ge \frac{n}{K_{r,\delta}}$ and $\frac{1}{\delta}\le K\le K_{r,\delta}$.

Let $s=2n^{2/3}$.
By Lemma~\ref{weyl}, 
\[
 \lambda_n(G)\ge\lambda_n(C)-\normop{R}
 \ge-a n^{1/4}-n^{2/3}\ge-2n^{2/3}=-s.
\]
By~\eqref{eq:structure-threshold}, $m\ge \frac{n}{K_{r,\delta}}\ge 13n^{2/3}>6s$.
Thus, Lemma~\ref{lem:obstruction} shows that $G$ does not contain three vertex subsets whose successive pairs are $\delta$-full and whose remaining pair is $\delta$-empty.

Let $Q$ be the graph on vertex set $\{v_1,\dots,v_K\}$ in which $v_iv_j\in E(Q)$ if and only if $(V_i,V_j)$ is $\delta$-full.
Suppose that $v_iv_jv_k$ is an induced $P_3$ of $Q$. Then as $v_iv_k\notin E(Q)$, $(V_i,V_k)$ is not $\delta$-full. Also, $(V_i,V_k)$ is not $\delta$-empty.  It must be $\delta$-exceptional. 
Each $\delta$-exceptional unordered pair can be the pair of endpoints of at most $K-2$ induced $P_3$. 
Lemma~\ref{lem:partition} consequently gives
$p_3(Q)\le\delta K^2\cdot K=\delta K^3$.
By Lemma~\ref{lem:editing}, $\ed_{c}(Q)\le 3\sqrt{\delta}K^2$. 

We obtain a new graph $G'$ from $G$ through the following operations. First, delete all edges incident to $V_0$. Next, add all missing edges inside each $V_i$ for $i=1,\dots,K$. Then, delete all edges between all $\delta$-empty pairs, and add all missing edges between $\delta$-full pairs. Finally, delete all edges between all $\delta$-exceptional pairs.
We now estimate the number of edges added or deleted in each step. 
The first step changes at most $|V_0|n\le\delta n^2$ edges. 
The second changes at most $K\binom m2\le\frac{(Km)^2}{2K}\le\frac{n^2}{2K}\le\frac{\delta}{2}n^2$.
The third changes at most $\delta\binom K2m^2\le\frac{\delta n^2}{2}$.
The final step changes at most $\delta K^2m^2\le\delta n^2$.
Thus, the total number of edges added or deleted is at most $3\delta n^2$.
Now, $G'[V_0]$ is independent, and $G'[V_i]$ is a clique for $i=1,\dots,K$. For $\{i,j\}\subseteq  \{1,\dots,K\}$,  $V_i$ and $V_j$ are completely connected in $G'$ precisely when their corresponding vertices in $Q$ are adjacent.

To conclude the construction, we obtain a new graph from $G'$ by performing the operations corresponding to the transformation of $Q$ into a union of cliques. 
Note that whenever an edge is added to or deleted from $Q$, we change all $m^2$ vertex pairs between the corresponding parts. Recall that $\ed_{c}(Q)\le 3\sqrt{\delta}K^2$.
The number of edges added or deleted is at most $3\sqrt{\delta}K^2 m^2\le 3\sqrt\delta n^2$. 
Together with the steps before, the total number of edges added or deleted is at most 
\[
 (3\delta+3\sqrt\delta)n^2
 =\left(\frac{3\eta^2}{100}+\frac{3\eta}{10}\right)n^2
 <\eta n^2.
\]
This shows that $\ed_{c}(G)\le \eta n^2$.
\end{proof}

Taking $R=0$ gives a structural statement for graphs with small least eigenvalue. For general $R$, one cannot discard the perturbation at the start: the bound $\normop{R}\le n^{2/3}$ is larger than the $n^{1/4}$ scale.

\subsection{Balanced clique partitions}

For a graph $G$, let $\ed_{bc}(G)$ be the minimum number of edge additions and deletions needed to turn $G$ into a balanced union of cliques. 

\begin{lemma}\label{lem:balance}
Let $G$ be a graph of order $n$ with average degree $k$.
Suppose that a union of cliques can be obtained from $G$ by $t$ edge additions or deletions. 
Let $\mathcal E=2t+\sum_{v\in V(G)}|d_G(v)-k|$.
If $\mathcal E<\frac{(k+1)n}{4}$, then $\ed_{bc}(G)\le  t+\frac{4\mathcal E n}{k+1}+\frac{n^2}{k+1}$.
\end{lemma}
\begin{proof}
Let $H$ be a union of cliques on $V(G)$ obtainable with $t$ edge additions or deletions, 
and let $a_1,\ldots,a_\ell$ be the orders of its components. Assume that $a_1,\dots,a_r\ge \frac{k+1}{2}$, and if $r<\ell$,  $a_{r+1},\dots,a_\ell<\frac{k+1}{2}$. 
For every vertex $v\in V(G)$, we have $|d_H(v)-k|=|d_H(v)-d_G(v)|+|d_G(v)-k|$ and so~\eqref{eq:degree-edit} gives \[
\sum_{v\in V(G)}|d_H(v)-k|\le2t+\sum_{v\in V(G)}|d_G(v)-k|\le \mathcal E.
\]
Let $W=\sum_{i=1}^\ell a_i|a_i-(k+1)|$. Then $W=\sum_{v\in V(G)}|d_H(v)-k|\le \mathcal E$.

For every vertex $v$ in a component of order smaller than $\frac{k+1}{2}$, it contributes more than $\frac{k+1}{2}$ to $W$. The number of such vertices is equal to $s=\sum_{i=r+1}^\ell a_i$. Then $s\le \frac{2\mathcal{E}}{k+1}$. As $\mathcal E<\frac{(k+1)n}{4}$, $s<n$. This implies $r\ge 1$.  
As $a_i\ge \frac{k+1}{2}$ for $i\le r$, we have $r\frac{k+1}{2}\le \sum_{i=1}^ra_i\le \sum_{i=1}^\ell a_i=n$ and so $r\le \frac{2n}{k+1}$. Moreover, \[
\frac{k+1}{2}\sum_{i=1}^r |a_i-(k+1)|\le \sum_{i=1}^ra_i|a_i-(k+1)|\le \sum_{i=1}^\ell a_i|a_i-(k+1)|\le \mathcal{E}.
\]
Thus $\sum_{i=1}^r |a_i-(k+1)|\le \frac{2\mathcal{E}}{k+1}$.
As $\sum_{i=1}^r a_i=n-s$, we have $|r(k+1)-n|\le\sum_{i=1}^r|(k+1)-a_i|+s\le\frac{4\mathcal E}{k+1}$.
It follows that
\[
 \sum_{i=1}^r\left|a_i-\frac nr\right|
 \le\sum_{i=1}^r|a_i-(k+1)|+r\left|(k+1)-\frac nr\right|
 \le\frac{6\mathcal E}{k+1}.
\]
Write $n=qr+r_0$, where $q=\lfloor \frac{n}{r}\rfloor$ and $0\le r_0<r$. Let $b_1,\ldots,b_r$ be such that $b_1=\cdots=b_{r_0}=q+1$ and
$b_{r_0+1}=\cdots=b_r=q$. Then \[
 \sum_{i=1}^r|a_i-b_i|\le \sum_{i=1}^r|a_i-\frac{n}{r}|+r\le \frac{6\mathcal E}{k+1}+r. 
\]

Let $H^*$ be a balanced union of cliques with exactly $r$ components. We obtain $H^*$ from $H$ by retaining $\min\{a_i,b_i\}$ vertices in each of the first $r$ components of $H$, and then assigning the remaining vertices among these $r$ components so that the $i$-th component has order $b_i$, then completing each component
and deleting all edges between distinct components. So the number of assigned vertices is \[
s+\sum_{i=1}^r(a_i-b_i)_+=s+\frac{1}{2}\sum_{i=1}^r(|a_i-b_i|+(a_i-b_i))=\frac{1}{2}\left(s+\sum_{i=1}^r|a_i-b_i|\right)\le \frac{4\mathcal E}{k+1}+\frac{r}{2}\le \frac{4\mathcal E}{k+1}+\frac{n}{k+1}. 
\]
From $H$ to $H^*$, the edges added or deleted are only those incident with assigned vertices. Hence the number of such edges is at most $\left(s+\sum_{i=1}^r(a_i-b_i)_+\right)n\le\frac{4\mathcal E n}{k+1}+\frac{n^2}{k+1}$. Thus, $\ed_{bc}(G)\le t+\frac{4\mathcal E n}{k+1}+\frac{n^2}{k+1}$.
\end{proof}

\section{From small discrepancy to a Tur\'an graph}\label{sec:main}

Throughout this section, let $G$ be a graph of order $n\ge2$, and let $p=p(G)$, $d=d(G)$, $D=\disc^+(G)$, and $M=A(G)-p(J-I)$.

\subsection{Two results on positive discrepancy}

\begin{lemma}\label{lem:degrees}
For a graph $G$ of order $n\ge 2$,  
$\sum_{v\in V(G)}|d_G(v)-d|\le2n\sqrt{2D}$.
\end{lemma}
\begin{proof}
Let $\mathbf{z}\in \R^n$ with $z_v=d_G(v)-d$ for $v\in V(G)$.
Then, as $p(n-1)=d$, $\mathbf{z}=M\mathbf{1}$. Let $Z=\{v\in V(G):z_v<0\}$ and let $s=-\sum_{v\in Z}z_v$. 
We have
$\sum_{v\in V(G)}z_v=\mathbf{1}^\top M\mathbf{1}=0$, and so $\sum_{v\in V(G)}|z_v|=2s$. 

If $s=0$, then $\mathbf{z}$ is a zero vector, proving the lemma. Suppose in the following that $s>0$. Let $t=\frac{s}{n^2}$. As $|z_v|\le n$, $0<t\le 1$. Let $\mathbf{x}=\mathbf{1}-t\mathbf{1}_Z$. Then $0\le x_v\le 1$ for $v\in V(G)$. Let $U\subseteq V(G)$ be the random subset in which each vertex $v$ is included independently with probability $x_v$.  Note that $\frac{1}{2}\mathbf{x}^\top M\mathbf{x}=\mathbb E\disc_G(U)\le D$, where $\mathbb E$ denotes expectation with respect to these independent choices.
On the other hand, \[
\mathbf{x}^\top M\mathbf{x}=-2t\mathbf{1}_Z^\top M\mathbf{1}+t^2\mathbf{1}_Z^\top M\mathbf{1}_Z=2ts+t^2\mathbf{1}_Z^\top M\mathbf{1}_Z\ge 2ts-t^2n^2=\frac{s^2}{n^2},
\]
where the inequality holds as $|\mathbf{1}_Z^\top M\mathbf{1}_Z|\le |Z|^2\le n^2$. So $s\le n\sqrt{2D}$.
Thus, $$\sum_{v\in V(G)}|d_G(v)-d|=\sum_{v\in V(G)}|z_v|=2s\le 2n\sqrt{2D},$$ as required.
\end{proof}

The second consequence is a matrix inequality. We derive it from Lemmas~\ref{lem:rounding} and~\ref{lem:positivepart}.
The minimization  problem  below uses the same quadratic semidefinite duality construction as in~\cite[Propositions~1 and~2]{KMR}. We give a direct proof to establish the bound $0\le R_{ii}\le2$ and the complementary-slackness relation $y_i(2-R_{ii})=0$ needed in our application.

\begin{lemma}\label{lem:certificate}
There exist a nonnegative diagonal matrix $Y=\diag(y_1,\ldots,y_n)$ and a matrix $R\succeq0$ such that $Y-M+R\succeq0$,  $R_{ii}\le2$ for any $i=1,\dots,n$ and $\sum_{i=1}^n y_i+\frac14\normF{R}^2\le4\pi D$.
\end{lemma}
\begin{proof}
Let $\beta=\max_{\sigma\in\{-1,1\}^n}\sigma^\top M\sigma$. For a sign vector $\sigma$, let $U$ be the set of vertices $v\in V(G)$ with $\sigma(v)=1$, and $U^c=V(G)\setminus U$. 
Write $\mathbf{x}=\mathbf 1_U$ and $\mathbf{z}=\mathbf 1_{U^c}$, so that $\sigma=\mathbf{x}-\mathbf{z}$.
Then
\[
0=(\mathbf{x}+\mathbf{z})^\top M(\mathbf{x}+\mathbf{z})
   =\mathbf{x}^\top M\mathbf{x}+2\mathbf{x}^\top M\mathbf{z}+\mathbf{z}^\top M\mathbf{z}.
\]
For the sign vector $\mathbf{x}-\mathbf{z}$ this gives \[
 (\mathbf{x}-\mathbf{z})^\top M(\mathbf{x}-\mathbf{z})=2\mathbf{x}^\top M\mathbf{x}+2\mathbf{z}^\top M\mathbf{z}=4\bigl(\disc_G(U)+\disc_G(U^c)\bigr)\le8D.
\]
Thus $\beta\le8D$.

Let $g$ be the function from Lemma~\ref{lem:positivepart}. For
$\mathbf{y}=(y_1,\ldots,y_n)$ with all $y_i\ge 0$, define
\[
\Phi(\mathbf{y})=\sum_{i=1}^n y_i+g\bigl(M-\operatorname{diag}(\mathbf{y})\bigr).
\]
By Lemma~\ref{lem:positivepart}, $g$ is continuous. Moreover, $\Phi(\mathbf{y})\ge\sum_i y_i\ge0$. The set $\{\mathbf{y}\in[0,\infty)^n:\Phi(\mathbf{y})\le\Phi(\mathbf{0})\}$ is nonempty, closed and contained in the compact cube $[0,\Phi(\mathbf{0})]^n$.
Hence $\Phi$ attains a minimum on $[0,\infty)^n$.
Fix a minimizing vector $\mathbf{y}$ of $\Phi$ and let $Y=\diag(\mathbf{y})$ and $R=(M-Y)_+$. Clearly, $R\succeq 0$. 
By Lemma~\ref{lem:positivepart}, for each $i$, we have
\[
 \frac{\partial\Phi}{\partial y_i}=1-\frac12R_{ii}.
\]
Note that $R_{ii}$ depends on $\mathbf{y}$. As $\mathbf{y}$ minimizes $\Phi$,  $\frac{\partial\Phi}{\partial y_i}=0$ if $y_i>0$ and $\frac{\partial\Phi}{\partial y_i}\ge 0$ if $y_i=0$. 
Consequently, $R_{ii}\le2$.
Moreover, $Y-M+R=(M-Y)_-\succeq0$.

Recall that $\mathbf{y}$ minimizes $\Phi$, $y_i(1-\frac{1}{2}R_{ii})=0$ for $i=1,\dots,n$. Then $y_iR_{ii}=2y_i$, and so $\ip{Y}{R}=\sum_{i=1}^ny_iR_{ii}=2\sum_{i=1}^ny_i$. 
Since $M-Y=R-(M-Y)_{-}$ and the positive and negative parts are orthogonal in the Frobenius inner product, $\langle M-Y,R\rangle=\|R\|_F^2$. Hence
\[
 \ip{M}{R}=\ip{Y}{R}+\normF{R}^2.
\]
Let $X=\frac{1}{2}R$. Then $X\succeq 0$ and $X_{ii}\le 1$ for $i=1,\dots,n$. By Lemma~\ref{lem:rounding}, we have \[
 \sum_{i=1}^n y_i+\frac14\normF{R}^2=\frac12\ip{Y}{R}+\frac14\normF{R}^2=\ip{M}{X}-\normF{X}^2 \le\frac\pi2\beta\le4\pi D,
\]
proving the lemma.
\end{proof}

In the proof of this lemma, $R$ is kept general. In particular, we do not claim the stronger statement that would be obtained by setting $R=0$.
This is why Theorem~\ref{thm:structure}, stated in the previous section, is formulated so as to accommodate a general matrix $R$.

\subsection{Proofs of the two main results}

Now, we are ready to prove Theorem~\ref{thm:discrepancy}.

\begin{proof}[Proof of Theorem~\ref{thm:discrepancy}]
Suppose that $G$ is $\eps$-far from every Tur\'an graph. 
Then $e(G)=\ed(G,T_1(n))\ge \eps n^2$ and $\binom n2-e(G)=\ed(G,T_n(n))\ge \eps n^2$. So $0<e(G)<\binom n2$ and $0<2\eps n^2\le\binom n2$, which implies $0<\eps<\frac{1}{4}$.
Let
\[
 \rho=\frac{\eps^2}{100},  \   \eta=\frac\rho4,  \   a_\eta=\frac{\eta^4}{4\cdot 10^9},  \  
 c_0=\frac{\eta a_\eta}{16\pi}.
\]
Let $N_\eta$ be an integer from Theorem~\ref{thm:structure}, and
\begin{equation}\label{eq:main-threshold}
 \begin{split}
 N_\eps=\left\lceil\max\left\{
 2N_\eta,\ \rho^{-1},\
 \left(\frac{4}{a_\eta}\right)^4,\
 \bigl(4\cdot2^{2/3}\sqrt{\pi c_0}\bigr)^{24},\
 \left(\frac{2\sqrt{2c_0}}{\rho}\right)^{8/3}
 \right\}\right\rceil.
 \end{split}
\end{equation}
Let $c=\min\{c_0, N_\eps^{-13/4}\}$.

First suppose $2\le n<N_\eps$.  Let $S$ be a set of two adjacent vertices of $G$. Then $D\ge \disc_G(S)=1-p$. Since $e(G)<\binom n2$, we have $1-p\ge \frac{1}{\binom n2}\ge \frac{2}{n^2}$. Then $c n^{5/4}\le n^{-2}\le D$. Hence $D\ge c n^{5/4}$, as desired.

Suppose in the following that $n\ge N_\eps$. 
Suppose to the contrary that $D<c_0n^{5/4}$. 

Let $A=A(\overline G)$ and choose $Y,R$ as in Lemma~\ref{lem:certificate}. Note that $M=A(G)-p(J-I)=J-I-A-p(J-I)=(1-p)(J-I)-A$. Then by Lemma~\ref{lem:certificate}, $A+R+Y+(1-p)I-(1-p)J=Y-M+R\succeq 0$. As $p\le 1$ and $J\succeq 0$, $(1-p)J\succeq 0$. Then 
\begin{equation}\label{eqq1}
 A+R+Y+(1-p)I\succeq0.
\end{equation}
Let $t_0=\frac{a_\eta n^{1/4}}{4}$ and $W=\{v:y_v>t_0\}$.
The bound $\sum_vy_v\le4\pi D$ yields
\[
 |W|\le\frac{\sum_vy_v}{t_0}
 \le\frac{16\pi D}{a_\eta n^{1/4}}
 <\frac{16\pi c_0}{a_\eta}n=\eta n.
\]
Let $U=V(G)\setminus W$ and $m=|U|$. Then
$m\ge(1-\eta)n\ge n/2\ge N_\eta$ and
$0\preceq Y[U]\preceq t_0I$. By Lemma~\ref{lem:certificate}, $R\succeq 0$ and $R_{vv}\le 2$ for $v\in U$.
Moreover, Lemma~\ref{lem:certificate} and~\eqref{eq:main-threshold} give
\[
 \normF{R[U]}\le\normF{R}\le4\sqrt{\pi D}
 \le4\sqrt{\pi c_0}\,n^{5/8}
 \le2^{-2/3}n^{2/3}\le m^{2/3}.
\]
By~\eqref{eqq1}, we have $(A+R+Y+(1-p)I)[U]\succeq 0$ and so
\[
 \lambda_m(A[U]+R[U])\ge-(t_0+1-p)\ge-(t_0+1)\ge -\frac{a_\eta}{2}n^{1/4}\ge -a_\eta |U|^{1/4}.
\]
By Theorem~\ref{thm:structure}, $\ed_c(\overline{G}[U])\le \eta |U|^2$.  
A union of cliques can be obtained from $\overline{G}$ by isolating all vertices of $W$ and adding or deleting at most $\eta |U|^2$ edges of $\overline{G}[U]$. Thus, \[
\ed_c(\overline{G})\le |W|n+\ed_c(\overline{G}[U])\le \eta n^2+\eta |U|^2\le 2\eta n^2\le \frac{\rho}{2}n^2.
\]

Let $k=d(\overline{G})$ and $\mathcal{E}=2\ed_c(\overline{G})+\sum_{v\in V(G)}|d_{\overline{G}}(v)-k|$.
As $n\ge N_\eps\ge \bigl(2\sqrt{2c_0}/\rho\bigr)^{8/3}$, we have $\sqrt{2c_0}\le (\frac{\rho n}{2})^{3/8}$.
For $v\in V(G)$, as $d_{\overline{G}}(v)-k=(n-1-d_G(v))-(n-1-d)=-(d_G(v)-d)$,
we have by Lemma~\ref{lem:degrees} that \[
\sum_{v\in V(G)}|d_{\overline{G}}(v)-k|=\sum_{v\in V(G)}|d_G(v)-d|\le 2n\sqrt{2D}\le 2\sqrt{2c_0}n^{13/8}\le \rho n^2.
\]
Recall that $\binom n2-e(G)\ge \eps n^2$. We have $e(\overline{G})\ge \eps n^2$ and so $k\ge 2\eps n$. Together with $\rho=\frac{\eps^2}{100}$, it then follows $k\ge 8\rho$. 
Thus,
\[
\mathcal{E}\le \rho n^2+\rho n^2<\frac{\eps}{2}n^2<\frac{(k+1)}{4}n.
\]
So by Lemma~\ref{lem:balance}, \[
\ed_{bc}(\overline{G})\le\ed_{c}(\overline{G})+\frac{4\mathcal{E} n}{k+1}+\frac{n^2}{k+1}\le \frac{\rho}{2}n^2+\frac{4\rho n^2}{\eps}+\frac{n}{2\eps}\le \frac{5\rho}{\eps}n^2=\frac{\eps}{20}n^2<\eps n^2.
\]
Note that the complement of the resulting balanced union of cliques is $T_\ell(n)$ for some integer $\ell$. This implies $\ed(G,T_\ell(n))<\eps n^2$, contradicting the fact that $G$ is $\eps$-far from every Tur\'an graph. Therefore, $D\ge c_0n^{5/4}\ge c n^{5/4}$. 
\end{proof}

\begin{proof}[Proof of Corollary~\ref{cor:eigenvalue}]
As $G$ is $d$-regular, $\overline G$ is $(n-1-d)$-regular. If $\overline{G}$ is a union of cliques, then every component of $\overline{G}$ has order $n-d$, and so $G$ is a Tur\'an\ graph, contrary to the fact that $G$ is $\eps$-far from every Tur\'an graph. Hence $\overline{G}$ is not a union of cliques.
By Lemma~\ref{lem:regular} and Theorem~\ref{thm:discrepancy},
\[
 c n^{5/4}\le\disc^{+}(G)
 \le\frac{2+\sqrt2}{8}\,n\lambda_2(G).
\]
Thus, $\lambda_2(G)\ge c'n^{1/4}$ with $c'=\frac{8c}{2+\sqrt2}=(8-4\sqrt2)c$.
\end{proof}

%
%

\appendix
\section{Appendix: Proof of Proposition~\ref{prop:lowrank}}\label{app:spectrum}

Throughout this appendix, $A$ is a symmetric $0$-$1$ matrix with zero
diagonal and $C=A+R$, where $R\succeq0$, $R_{ii}\le2$ for $i=1,\dots,n$, $\normF{R}\le n^{2/3}$ and $\lambda_n(C)\ge-a n^{1/4}$. 
We have $0\le \tr C=\tr R\le 2n$ and $\normF{A}^2=\sum_{i,j}A_{ij}^2\le n^2$. So by triangle inequality, $\normop{C}\le \normF{C}\le \normF{A}+\normF{R}\le n+n^{2/3}\le 2n$. 

Write the eigenvalues of $C$ as $\theta_1\ge\cdots\ge\theta_n\ge-s$, where $1\le s\le n^{1/4}$, and let $\mathbf{z}_1,\dots,\mathbf{z}_n$ be the corresponding orthonormal eigenvectors.
Then
\begin{equation}\label{eq:basic-spectral}
 \begin{gathered}
 \sum_{\theta_i>0}\theta_i
 =\tr C+\sum_{\theta_i<0}|\theta_i|\le2n+ns\le3ns.
 \end{gathered}
\end{equation}
 For $t>0$, let
\[
 I(t)=\{i\in[n]:\theta_i\ge t\},  \  
 S(t)=\sum_{i\in I(t)}\theta_i.
\]
Thus $|I(t)|\le\frac{S(t)}{t}$. 

Let $N_*=104^{30}$. For $n\ge N_*$, we have
$5n^{-2/15}\le\frac{1}{4}$, $26n^{-1/30}\le\frac{1}{4}$ and $16n^{7/8}\le2n$.

\begin{lemma}\label{lem:compression}
Let $n\ge N_*$, $q\in\{2,3\}$ and $t>0$ satisfying $t\ge16s\,n^{1-1/q}$. If $tS(t)\ge n^{9/5}$, then $S\left(\frac{t^q}{(16n)^{q-1}}\right)\ge\frac{S(t)^q}{4n^{q-1}}$.
\end{lemma}
\begin{proof}
Let $X=\sum_{i\in I(t)}\theta_i \mathbf{z}_i \mathbf{z}_i^\top$.
Then $0\preceq X\preceq C+sI$. By~\cite[Theorem~4.1]{JMTZ}, $(C+sI)^{\circ q}\succeq X^{\circ q}$. 
Let $W=\spanof\{\mathbf{z}_{i_1}\circ\cdots\circ \mathbf{z}_{i_q}: i_1,\ldots,i_q\in I(t)\}$ and $d=\dim W$.
Then $d\le|I(t)|^q\le\bigl(\frac{S(t)}{t}\bigr)^q$.
By the identity $(\mathbf{x}\mathbf{x}^\top)\circ(\mathbf{y}\mathbf{y}^\top)=(\mathbf{x}\circ \mathbf{y})(\mathbf{x}\circ \mathbf{y})^\top$, each column of $X^{\circ q}$ is a Hadamard product of $q$ columns of $X$. Hence the range of $X^{\circ q}$ is contained in $W$. 
So H\"older's inequality gives
\[
\tr_W\bigl((C+sI)^{\circ q}\bigr)\ge\tr X^{\circ q}=\sum_{i=1}^n X_{ii}^q\ge\frac{(\tr X)^q}{n^{q-1}}=\frac{S(t)^q}{n^{q-1}}.
\]

Decompose $(C+sI)^{\circ q}$ as $(C+sI)^{\circ q}=C+D+E$, where $D$ is the diagonal part of $(C+sI)^{\circ q}-C$ and $E$ is its zero-diagonal part. 
As $0\le C_{ii}\le2$ and $s\ge1$, we have $0\preceq D\preceq3^qs^qI$ and so $\tr_W D\le3^qs^q d$.
For $i\ne j$, let $f_q(t)=t^q-t$. Then $f_q(A_{ij})=0$ and $E_{ij}=f_q(C_{ij})-f_q(A_{ij})$.
Let $L_q=\sup_{x\in[-2,3]}|f_q'(x)|$. Then $L_2=5$ and $L_3=26$. Hence, by the mean value theorem, $\normF{E}\le L_q\normF{R}\le L_q n^{2/3}$.  
By Cauchy--Schwarz inequality, $|\tr_W E|\le \normF{P_W}\normF{E}\le  L_q n^{2/3}\sqrt d$.

Let $K=\frac{t^q}{(16n)^{q-1}}$. 
The proof of~\cite[Lemma~6.2]{JMTZ}, which applies directly to any real symmetric matrix, yields $\tr_W C\le S(K)+Kd$. Thus,
\[
 S(K)\ge\tr_W C-Kd\ge \tr_W\bigl((C+sI)^{\circ q}\bigr)-\tr_W D-\tr_W E-Kd\ge \frac{S(t)^q}{n^{q-1}}
       -(K+3^qs^q)d-L_q n^{2/3}\sqrt d.
\]
As $S(t)>0$, \[
\frac{(K+3^qs^q)d+L_q n^{2/3}\sqrt d}{\frac{S(t)^q}{n^{q-1}}}
\le
\frac{1}{16^{q-1}}+\left(\frac{3}{16}\right)^q
+\frac{L_q n^{q-1+2/3}}{(tS(t))^{q/2}}\le \frac{1}{16}+\frac{9}{256}+\frac{1}{4}<\frac{3}{4}.
\]
Thus, $S(K)\ge \frac{S(t)^q}{4n^{q-1}}$.
\end{proof}

For $n\ge N_*$, the case $q=2$ also gives
\begin{equation}\label{eq:quadratic-alternative}
 S(t)\le\max\left\{
  2\sqrt{nS\left(\frac{t^2}{16n}\right)},
  \frac{n^{9/5}}{t}\right\}
   \  \text{for }t\ge16s\sqrt n.
\end{equation}

\begin{lemma}\label{lem:tail}
Let $n\ge N_*$ and $t>0$ satisfying $16n^{7/8}\le t\le2n$. Then $S(t)\le4n\left(\frac{16n}{t}\right)^{3/4}$.
\end{lemma}
\begin{proof}
Let $a$ be such that $0<a\le\frac18$ and $t=16n^{1-a}$, and let
$F(x)=\frac{S(x)}{4n}$ for $x>0$.
It suffices to prove that $F(t)\le n^{3a/4}$.

Suppose to the contrary that $F(t)>n^{3a/4}$.
We will carry this inequality forward along a sequence of parameter transformations until the parameter lies in $[\frac13,\frac12]$, where it contradicts the known upper bound. To this end, we first choose a sequence of transformations that takes the parameter $a$ to a number in $[\frac13,\frac12]$.

Let $j$ be the least nonnegative integer such that $2^ja\ge\frac13$.
Then $j\ge2$ and $\frac13\le2^ja<\frac23$.
If $2^ja\le\frac12$, then let $b=2^j$. 
If $2^ja>\frac12$, then $j\ge3$ and the minimality of $j$ gives
$a<\frac1{3\cdot2^{j-1}}\le\frac1{12}$. 
In this case,  let $b=3\cdot2^{j-2}$. 
The final parameter is then
$$ba=\frac34\,2^ja\in\left(\frac38,\frac12\right).$$
Thus, in either case, the final parameter $ba$ lies in
$[\frac13,\frac12]$.
For $q=2$, the starting value is at most $\frac{ba}{2}\le\frac14$.
For $q=3$, the starting value is $a<\frac1{12}$.
Thus the required bounds on $a$ hold in all cases.

Now we show by induction that $F(16n^{1-c})>n^{3c/4}$ for every $c$ in the chosen sequence.
The assertion holds at the start $c=a$ by our contradiction hypothesis.
Suppose it holds at some $c$ from which the next step is to be made, and let $x=16n^{1-c}$.
Since $c\le ba\le\frac12$, we have
$$xS(x)=64n^{2-c}F(x)>64n^{2-c/4}
      \ge64n^{15/8}>n^{9/5}.$$
By Lemma~\ref{lem:compression},
\[
F\left(16n^{1-qc}\right)
=\frac{S\left(16n^{1-qc}\right)}{4n}
=\frac{S\left(\frac{(16n^{1-c})^q}{(16n)^{q-1}}\right)}{4n}
\ge \frac{S(16n^{1-c})^q}{16n^q}
=4^{q-2}F(16n^{1-c})^q
\ge F(16n^{1-c})^q.
\]

At the end $ba$, however,~\eqref{eq:basic-spectral} gives $$n^{1/4}\le n^{3ba/4}<F(16n^{1-ba}) \le\frac{3s}{4}\le\frac34 n^{1/4},$$
a contradiction.
It follows that our contradiction hypothesis was false, and hence $F(t)\le n^{3a/4}$.
Therefore, $S(t)\le4n\,n^{3a/4}=4n\left(\frac{16n}{t}\right)^{3/4}$, as required.
\end{proof}

\begin{proof}[Proof of Proposition~\ref{prop:lowrank}]
Let $a=\frac{\tau}{1000}$ and $\kappa=(\frac{\tau}{1024})^4$.
Let
\[
 N^*(\tau)=\left\lceil\max\left\{
 N_*,\ a^{-20},\ \left(\frac{16}{\kappa}\right)^8,
 \left(\frac{256}{\tau}\right)^{16}
 \right\}\right\rceil.
\]
Assume $n\ge N^*(\tau)$ and let $s=an^{1/4}$.
Then $\theta_n\ge-s$ and $n^{1/5}\le s\le n^{1/4}$.
Let
$t_0=16s\sqrt n$, $t_1=16\sqrt s n^{3/4}$ and $L=16n^{7/8}$.
Then $t_0\le t_1\le L\le\kappa n<2n$.
\eqref{eq:basic-spectral} gives
\[
 \sum_{\theta_i<0}\theta_i^2
 +\sum_{0<\theta_i<t_0}\theta_i^2
 \le ns^2+3nst_0
 \le n^{3/2}+48s^2n^{3/2}.
\]
Note that $\frac{t_0^2}{16n}=16s^2$ and
$\frac{t_1^2}{16n}=t_0$. Applying~\eqref{eq:quadratic-alternative} at $t_0$ and $t_1$, respectively, gives
\begin{align*}
 S(t_0)&\le\max\left\{\sqrt{12}\,n\sqrt s,
                       \frac{n^{9/5}}{t_0}\right\}
          \le4n\sqrt s,\\
 S(t_1)&\le\max\left\{2\sqrt{nS(t_0)},
                       \frac{n^{9/5}}{t_1}\right\}
          \le4ns^{1/4}.
\end{align*}
As $s\ge n^{1/5}$, we have
\[
 \frac{n^{9/5}}{t_0}\le\frac{n\sqrt s}{16},  \  
 \frac{n^{9/5}}{t_1}\le\frac{n^{19/20}}{16}
                     \le\frac{ns^{1/4}}{16}.
\]
Consequently,
\begin{align*}
 \sum_{t_0\le\theta_i<t_1}\theta_i^2
 &\le t_1S(t_0)\le64s n^{7/4},\\
 \sum_{t_1\le\theta_i<L}\theta_i^2
 &\le LS(t_1)\le64s^{1/4}n^{15/8}\le64n^{31/16}.
\end{align*}
Finally, Lemma~\ref{lem:tail} gives
$S(t)\le32n^{7/4}t^{-3/4}$ for $L\le t\le\kappa n$.
Summing the identity
$\theta^2=L\theta+\int_L^\theta\theta\,dt$ over
$L\le\theta_i<\kappa n$ yields
\[
 \sum_{L\le\theta_i<\kappa n}\theta_i^2\le LS(L)+\int_L^{\kappa n}S(t)\,dt
\le32n^{7/4}L^{1/4}
       +128n^{7/4}\bigl((\kappa n)^{1/4}-L^{1/4}\bigr)\le128\kappa^{1/4}n^2.
\]
Note that $s=an^{1/4}$, $48a^2+64a+128\kappa^{1/4}
\le\frac{112\tau}{1000}+\frac{\tau}{8}<\frac{\tau}{4}$, $64n^{-1/16}\le\frac{\tau}{4}$ and $n^{-1/2}\le\frac{\tau}{4}$. 
Then
\[
 \sum_{\theta_i<\kappa n}\theta_i^2
 \le(48a^2+64a+128\kappa^{1/4})n^2
      +64n^{31/16}+n^{3/2}
 \le\tau n^2.
\]

Let $B=\sum_{\theta_i\ge\kappa n}\theta_i \mathbf{z}_i \mathbf{z}_i^\top$.
Then $B\succeq0$ and
\[
 \normF{C-B}^2=\sum_{\theta_i<\kappa n}\theta_i^2\le\tau n^2,
   \  
 \rank B\le\frac{\normF{C}^2}{(\kappa n)^2}\le4\kappa^{-2}.
\]
Also $\normop{B}\le\normop{C}\le2n$, completing the proof.
\end{proof}

\section*{Declaration on the use of AI}

During the preparation of this manuscript, one author used ChatGPT 5.6 as an auxiliary tool, primarily to check technical details, and to improve the language of the manuscript.

\end{document}